\documentclass[12pt]{amsart}

\usepackage[latin1]{inputenc}
\usepackage{subfig}
\usepackage{amsmath,amssymb}
\usepackage{mathtools}
\usepackage{stmaryrd}
\usepackage[]{caption}
\usepackage{todonotes}
\usepackage{float}
\usepackage{todonotes}

\newtheorem{Theorem}{Theorem}[section]
\newtheorem{proposition}[Theorem]{Proposition}
\newtheorem{corollary}[Theorem]{Corollary}

\newtheorem{definition}[Theorem]{Definition}

\newcommand{\Z}{\mathbb Z}

\newcommand{\A}{\mathcal{A}}
\newcommand{\B}{\mathcal{B}}

\usepackage[backend=bibtex]{biblatex}
\title{Generating series of Periodic Parallelogram Polyominoes}
\author[J. C. Aval, A. Boussicault, P. Laborde-Zubieta, M. P\'etr\'eolle]{Jean-Christophe Aval, Adrien Boussicault, Patxi Laborde-Zubieta, Mathias P\'etr\'eolle }
\address{LaBRI - Universit\'e de Bordeaux, 351 cours de la Lib\'eration F-33405 Talence cedex}

\begin{document}
\maketitle

\begin{abstract}
The aim of this work is the study of the class of periodic parallelogram polyominoes, and two of its variantes. These objets are related to 321-avoiding affine permutations. We first provide a bijection with the set of triangles under Dyck paths. We then prove the ultimate periodicity of the generating series of our objects, and introduced a notion of primitive polyominoes, which we enumerate. We conclude by an asymptotic analysis.

\end{abstract}

\section{Introduction}

The study of polyominoes is very classical in combinatorics.
Many classes of these objects have been considered and studied in the past decades.
The focus of the present work is the class of {\em periodic parallelogram polyominoes}
(nicknamed as PPP's), which may be seen as parallelogram polyominoes drawn on a cylinder.
These objects were introduced simultaneously and independently in \cite{BBMJN} and \cite{BLZ} .
Their introduction in \cite{BBMJN} is linked to the study of
the bivariate generating series of 321-avoiding affine permutations with respect to the rank and the Coxeter length (or equivalently, to the study of fully commutative elements with full support in Coxeter groups of type $\widetilde{A}_{n-1}$); 
PPP's are enumerated in this paper  through the use of heaps of segments (extending the case of parallelogram polyominoes).
In the present work, which is intended as a sequel to \cite{BLZ},
we use a tree structure (inspired by \cite{BRS}) to study PPP's.
This structure puts to light a new parameter,
called {\em intrinsic  thickness} (see Definition \ref{def:it}).
We give here several answers raised in \cite{BLZ}.
In Section \ref{sec:bijPPP}, we give a bijective explanation
to a fact noticed and proved in \cite{BLZ}: the number
of PPP's with fixed intrinsic thickness according to their semi-perimeter coincides with the sequence {\tt A008549} in \cite{oeis}, which gives the total area under Dyck paths.
In Section \ref{sec:periodic}, we study the area parameter in PPP's and
we prove a periodicity property for the coefficients of the corresponding generating series.
Moreover, this leads to the introduction of a notion of {\em primitive} PPP's,
for which we obtain a very simple enumerative formula (see Theorem~ \ref{thmprimitive} -
a bijective proof is still to be found).
To conclude, we give in Section \ref{sec:asymptotic} an asymptotic estimate of the coefficients of the generating function of {\em strips} 
(these objects, introduced in  \cite{BLZ}, may be defined as orbits of PPP's
under the rotation of the cylinder), according to their semi-perimeter.

\section{Preliminaries}

First of all, we recall the main definitions and results from \cite{BLZ} we
will need in this article. The definition of periodic parallelogram polyominoes
is based on \emph{parallelogram polyominoes}. They can be seen as a maximal
set of cells of $\Z\times\Z$ defined up to translation, contained in between
two paths with North and East steps that intersect only at their starting and
ending points. In the following, the \emph{first} column will correspond to the
leftmost one, and the \emph{last} column to the rightmost one.

\begin{definition}[PPP]
    A \emph{periodic parallelogram polyomino} is a parallelogram polyomino $P$
    with a positive integer $g$ called the \emph{gluing size}, such that $g$ is
    smaller or equal to the minimum of the heights of the first and last column
    of $P$. Moreover, in the rest of the article we consider PPP's which are
    not of rectangular shape with gluing size equal to the size of the columns.
\end{definition}

We represent the integer $g$ of a PPP with a marking in the leftmost and
rightmost columns. This marking indicates how we glue the first and the last
columns, as we can see in Figure~\ref{fig_PPP}. In the following, two rows
glued together count as the same row. Let us define some statistics about
PPP's, the number of column is called \emph{width}, the number of rows, i.e.
the number of rows strictly below the marking of the last column, is called the
\emph{height}, the \emph{semi-perimeter} corresponds to the sum of the height
and the width, and the number of cells is the \emph{area}. For example, the
PPP in Figure~\ref{fig_PPP} is of height 5, width 8, semi-perimeter 13 and
area 26.

\begin{figure}[ht]
    \begin{center}
        \includegraphics[scale=.35]{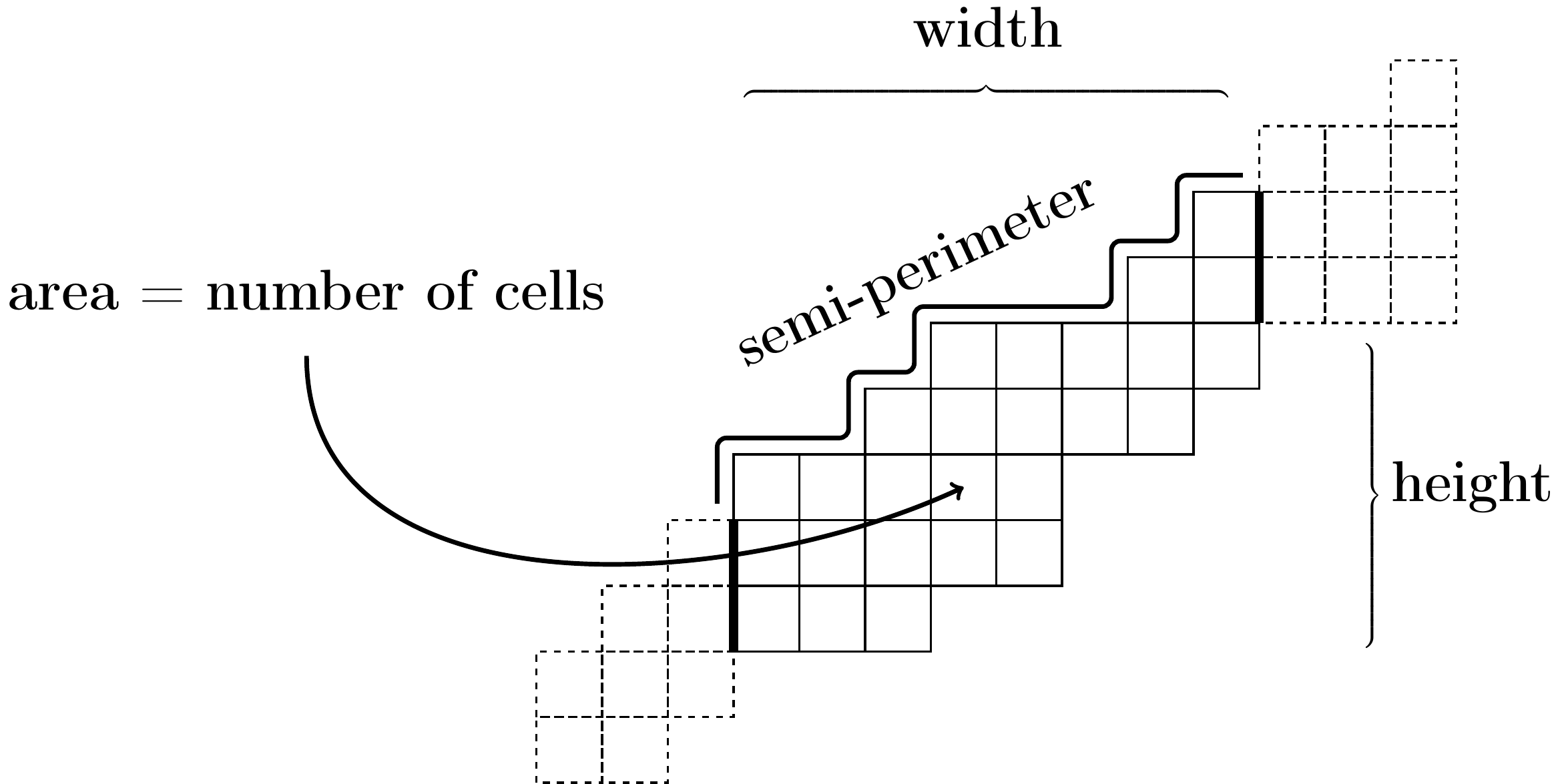}
    \end{center}
    \caption{A PPP and the illustration of the gluing (dotted).}
    \label{fig_PPP}
\end{figure}

As introduced in \cite{BLZ}, for each PPP $P$, we construct a rooted map
$\varphi(P)$ as follows. The vertices correspond to the top cell of each column
(column vertices) and the rightmost cell of each row (row vertices). We put an
edge between each vertex and its parent: if a vertex is column vertex, its parent is
the row vertex which belongs to the same row, and the parent of a row vertex is
the column vertex which belongs to the same column. The connected components of
the graph obtained this way, have exactly one cycle. We will now embed the
graph in the plane. If we orient each edge of the cycles from the child to its
parent, we embed the cycles in the plane clockwise. Moreover, we order counter
clockwise the children of a vertex $v$ with respect to their distance with $v$
starting with the closest one, the father of $v$ being between the two extremal
children. Finally, we root the map in the column vertex corresponding to the
first column of $P$. An example is given in Figure~\ref{fig:phi}. It should be
noted that the cycles of $\varphi(P)$ have the same even length, moreover,
since there is an alternation between column vertices and row vertices, we can
bicolor $\varphi(P)$ in black (column vertices) and white (row vertices).

\begin{figure}[ht]
    \begin{center}
        \includegraphics[scale=0.25]{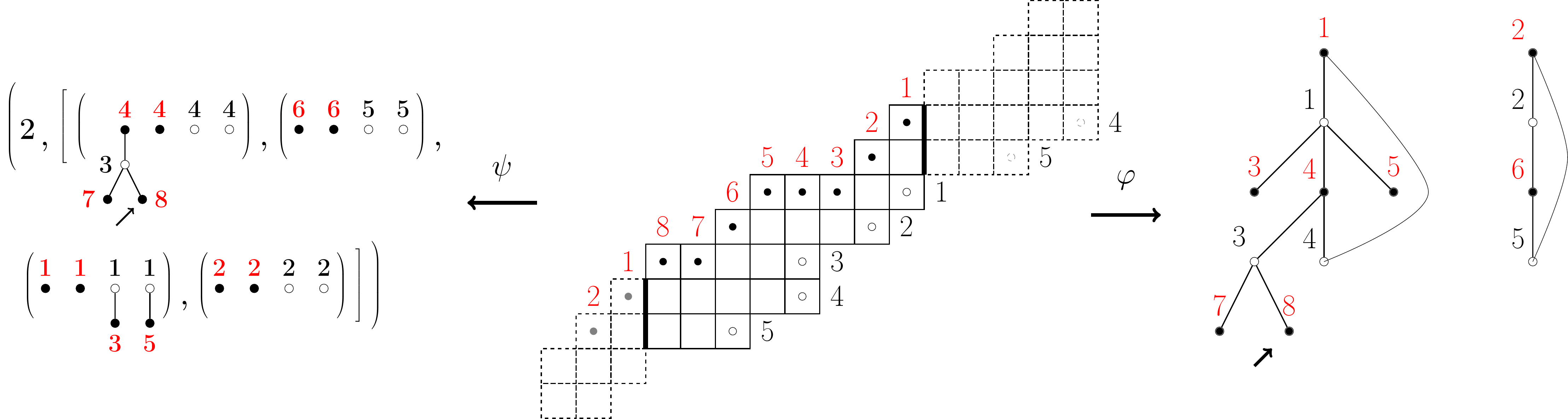}
    \end{center}
    \caption{The maps $\varphi$ and $\psi$.}

    \label{fig:phi}
\end{figure}

\begin{definition}
    A PPP $P$ is called a \emph{trunk PPP} if $\varphi(P)$ is a disjoint
    union of cycles. Trunk PPP's are the ones such that the upper path is of
    the form $N^k(NE)^l$, the lower path is of the form $(EN)^lN^k$ and the
    gluing size is equal to $k$, with $l$ and $k$ two positive integers.
\end{definition}
Let $P$ be a PPP, the leaves of $\varphi(P)$ correspond to the column or the
rows of $P$ containing only one vertex. By removing recursively the rows and
the columns corresponding to leaves, we obtain a trunk PPP noted $trunk(P)$.
\begin{definition}\label{def:it}
    We call \emph{intrinsic thickness} of a PPP $P$ the gluing size of
    $trunk(P)$.
\end{definition}

If we label the row vertices and the column vertices of a trunk PPP $P$ both
from 1 to $l$, we can enrich $\varphi(P)$ with this labelling, we call it the
\emph{cyclic structure} of $P$. For a general PPP $P$, its cyclic structure
corresponds to the cyclic structure of $trunk(P)$.

The fact that two trunk PPP's can have the same image with respect to
$\varphi$, shows that $\varphi$ is not injective. But if we only consider
PPP's of intrinsic thickness equal to one, we have the following result.
\begin{proposition}\label{prop:bij_it_1}
    The map $\varphi$ gives a bijection between PPP's of intrinsic thickness
    equal to one and connected rooted maps containing exactly one cycle of even length.
\end{proposition}
\begin{proof}
    We just need to notice that $\varphi$ gives a bijection between trunk
    PPP's of intrinsic thickness equal to 1 and cycles of even size, and use
    the pruning define in \cite[Section~4]{BLZ}. The bicoloring of the rooted
    map is not necessary since coloring in black the root induces the coloring
    of the rest of the map.
\end{proof}

In the general case we have the following result \cite[Theorem~7.1]{BLZ}.
\begin{Theorem}\label{thm:bijquadruplet}
    The PPP's are in bijection with pairs composed of a positive integer (the
    intrinsic thickness) and a non empty list of 4-tuples of bicolored ordered
    trees such that:

$\bullet$ each 4-tuple is composed of two black rooted trees and two white
            rooted trees,

$\bullet$ in the first 4-tuple we mark a non-root black vertex or the two
            black roots.\\
We denote $\psi$ this bijection. In the rest of the article, the expression ``a
list of 4-tuples of trees" should be understood as a non empty list of 4-tuples
satisfying the previous conditions. 
\end{Theorem}

Let $P$ be a PPP, the vertices of $trunk(P)$ correspond exactly to the vertices
of $\varphi(P)$ which compose the cycles. More precisely, the column vertex and
the row vertex contained in a same column correspond to two consecutive
vertices in a cycle of $\varphi(P)$. In the previous bijection, the four trees
of each tuple correspond to the trees rooted in each two consecutive cycle
vertices of $\varphi(P)$, two of them belong to the inner face of the cycle,
and the other two, to the outer face. An example of $\psi$ is given in
Figure~\ref{fig:phi}.

We will also deal with two other objects derived from PPP's. The marked PPP's
are in bijection with the fully commutative affine permutations (\cite{BBMJN}),
that is, affine permutations avoiding 321. Regarding strips, they appeared in
the study of PPP's in \cite{BLZ}.
\begin{definition}
    A \emph{marked PPP} is a PPP with a marking, in one of its horizontal edges
    of the first column which are above the gluing, including the one at the
    top of the gluing.

    Due to the periodic structure of PPP's, we can define a rotation on PPP's
    which induces a partitioning in equivalent classes called \emph{strips}.
\end{definition}
For example, there are 2 possibilities to mark the PPP of Figure~\ref{fig_PPP}.

\section{Bijection with the set of triangles under Dyck paths}\label{sec:bijPPP}

The aim of this section is to provide a bijective explanation for the following result,
proved analytically in \cite{BLZ}.

\begin{proposition}
\label{pop_dyck_area}
The number of PPP's with intrinsic thickness fixed to 1 and half-perimeter equal to $n$
is given by $4^{n-1} - {2n-1 \choose n-1}$, which may be interpreted (\cite{oeis})
as the total (triangular) area under Dyck paths of size $n-1$.
\end{proposition}

\begin{definition}
Let $\A_n$ denote the set of triples $(A,s,i)$ where 

$\bullet$  $A$ is a (rooted) planar tree with $n$ vertices,

$\bullet$  $s$ is a vertex of $A$ different from its root,

$\bullet$  $i$ is an integer, $1\le i\le 2p(s)-1$ where $p(s)$ is the depth of vertex $s$.
\end{definition}
See Figure \ref{fig:bij1} for an illustration.

It is easy to see that the sequence of cardinalities of $\A_n$ coincides with  {\tt A008549} in \cite{oeis}.

\begin{definition}
Let $\B_n$ denote the set of couples $(B,r)$ where

$\bullet$  $B$ is a connected planar graph with $n$ vertices, having exactly one cycle of even length,

$\bullet$  $r$ is a vertex of $B$.
\end{definition}

Another way to present it is as follows.
An element of $\B_n$ is given by a (planar) cycle of length $2\ell$, 
with $4\ell$ planar trees attached (one for each vertex of the cycle, in the inner and outer face of the cycle); among the whole set of $n$ vertices, one is distinguished.

Because of Proposition~\ref{prop:bij_it_1}, the set $\B_n$ is in bijection with PPP's of semi-perimeter $n$,
and whose intrinsic thickness is equal to 1.

\begin{proposition}\label{thm:bijPhi}
The sets $\A_n$ and $\B_n$ are in bijection.
\end{proposition}

\begin{proof}

We shall construct a bijection $\Phi$ between $\A_n$ and $\B_n$.
Let $(A,s,i)$ be an element of $\A_n$ (Figure \ref{fig:bij1}).
We shall create a cycle by adding an edge from $s$ or from its parent (denoted $s'$) 
to one of its ancestor vertices $t$ ($t$ belongs to the path from the root to $s$).
This new edge $e$ will be fixed to $t$ either to its left or right when $t$ is different from the root,
with just one possibility when $t$ is the root: we thus have exactly $2p(s)-1$ possibilities, 
which we label by integers counterclockwise.
(see Figure \ref{fig:bij1}; on this example the vertex $s$ has depth equal to 4, whence 7 possibilities).

\begin{figure}[ht]
\begin{center}
   \includegraphics[scale=0.45]{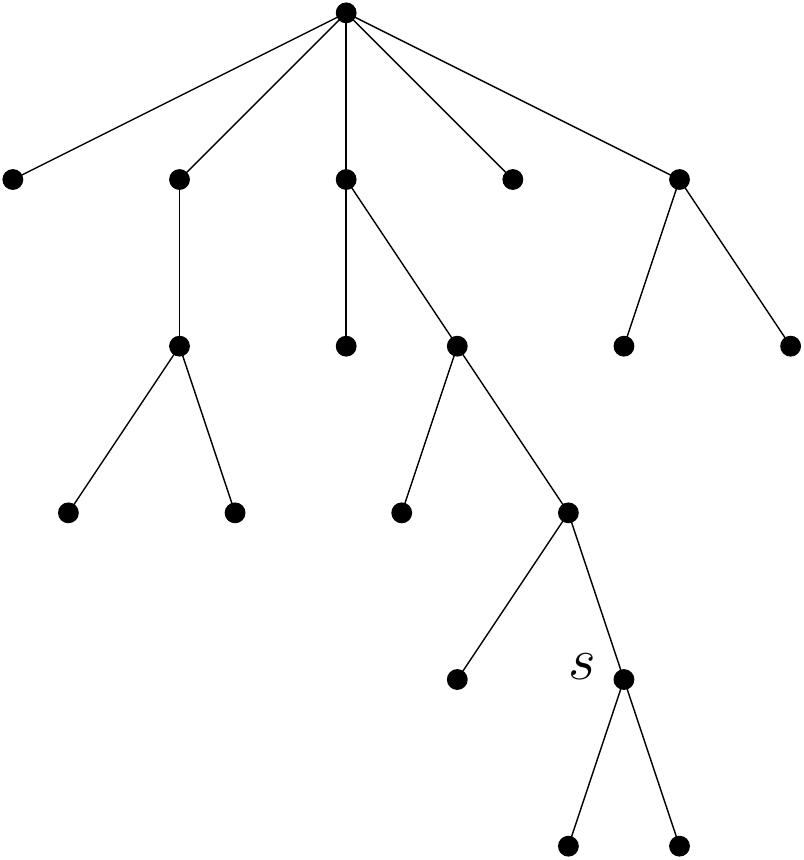}
   \hspace{1cm}
   \includegraphics[scale=0.45]{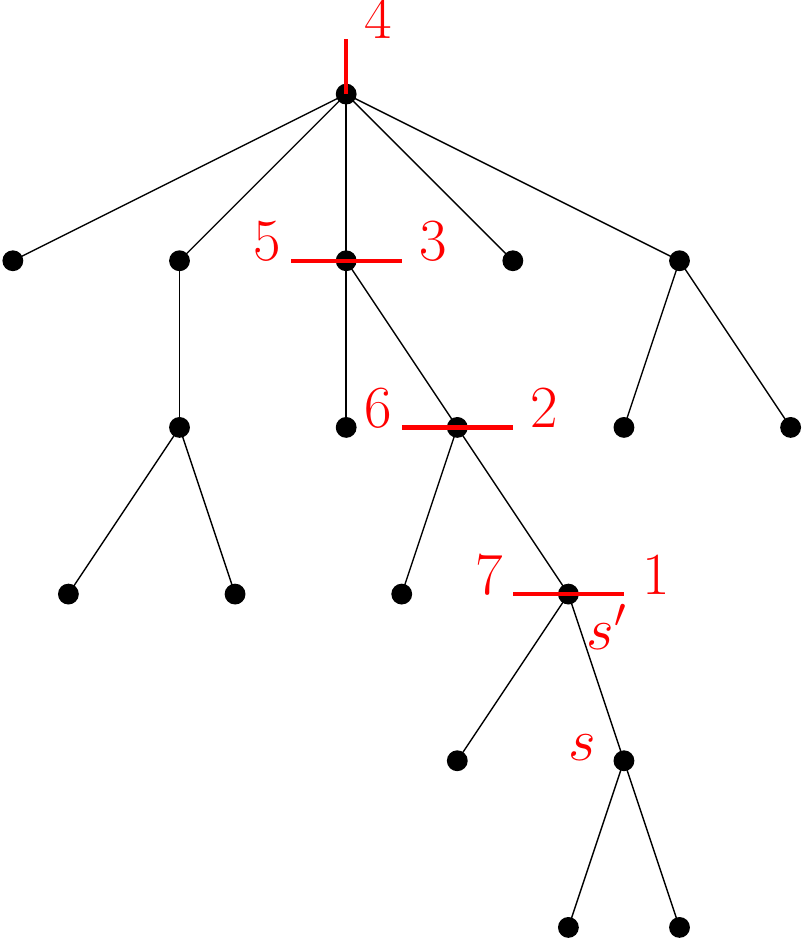}
\end{center}
\caption{An element of $\A_{19}$ and the labelling of the $2p(s)-1=7$ possibilities for $i$.}
\label{fig:bij1}
\end{figure}
We then distinguish two cases according to the integer $i$ being odd or even:

$1.$ if $i$ is odd, the edge $e$ goes from the right of $s$ to the half-edge labelled by $i$;

$2.$ if $i$ is even, $e$ goes from $s'$, plugged at the left of the edge of $s$, to the half-edge labelled by $i$.\\
In any case, the edge $e$ goes around the tree counterclockwise. A last step in
the construction consists in cutting the subtree of the root located at the
right of the edge going towards $s$ (i.e. from the root to its son which is an
ancestor of $s$), and to plug it on the vertex $t$:

$\bullet$  just after (counterclockwise) the edge $e$ if $i<p(s)$,

$\bullet$  just before (counterclockwise) the edge $e$ if $i>p(s)$\\
(we do nothing if $e$ is plugged on the root, i.e. for $i=p(s)$).
We then distinguish the root $r$.
We have built in this way the image $\Phi(A,s,i)$.
See Figures \ref{fig:bij2} ($i=6$) and \ref{fig:bij3} ($i=3$).
\begin{figure}[ht]
\begin{center}
   \includegraphics[scale=0.45]{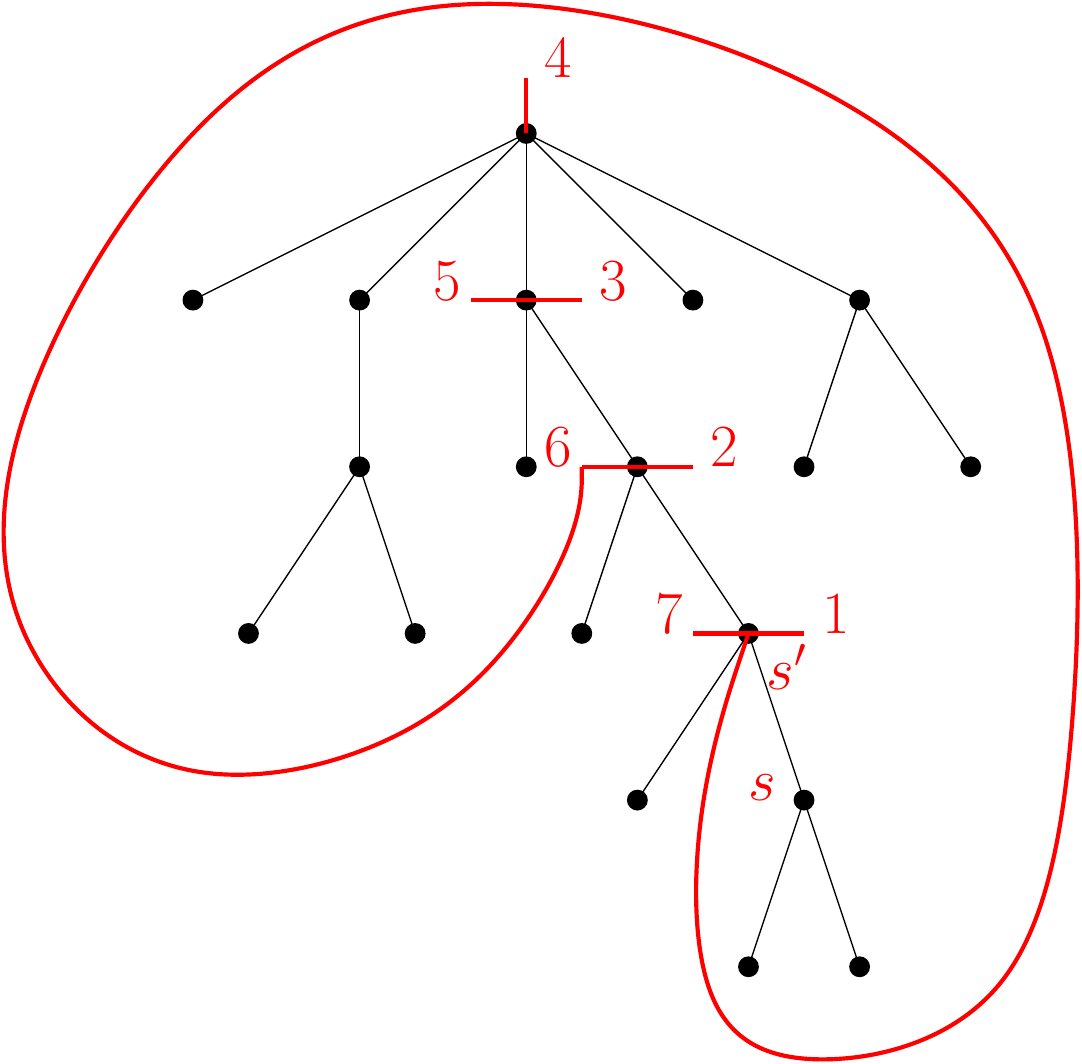}
   $\ \longrightarrow \ $
   \includegraphics[scale=0.45]{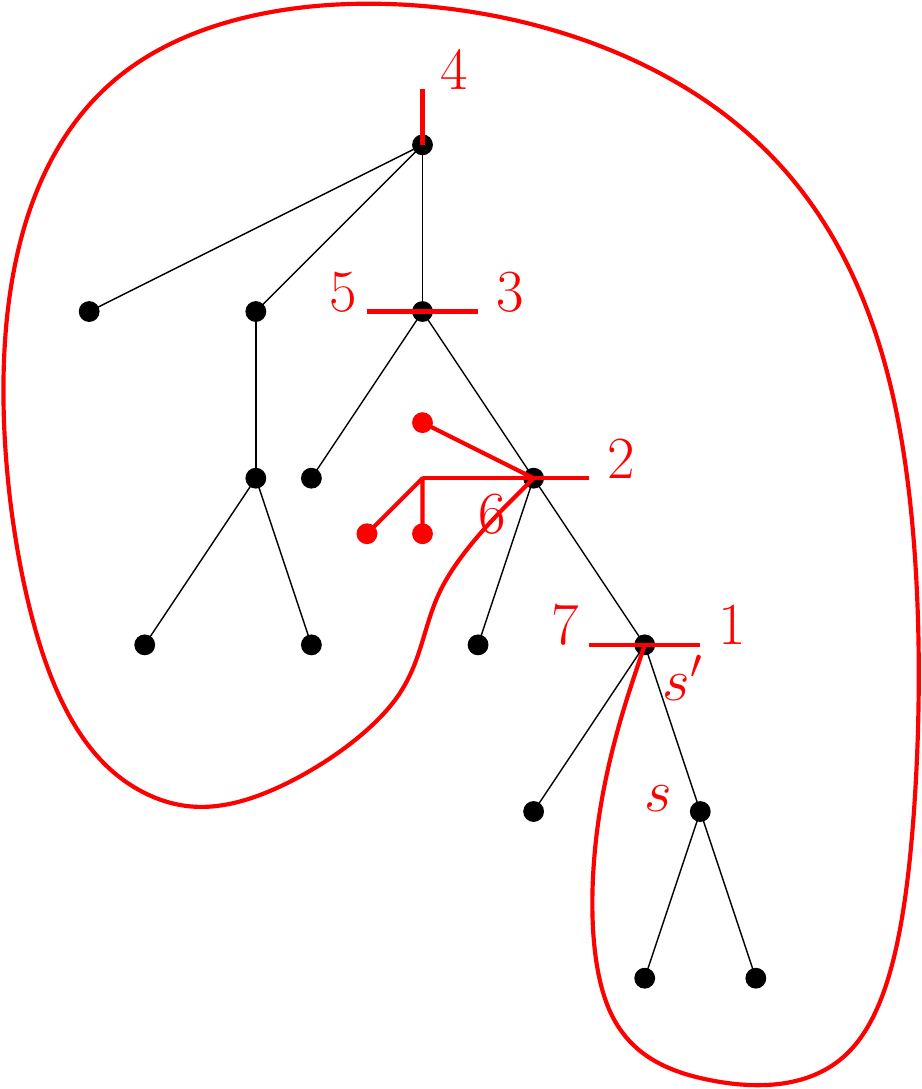}
   $\ \longrightarrow \ $
   \includegraphics[scale=0.45]{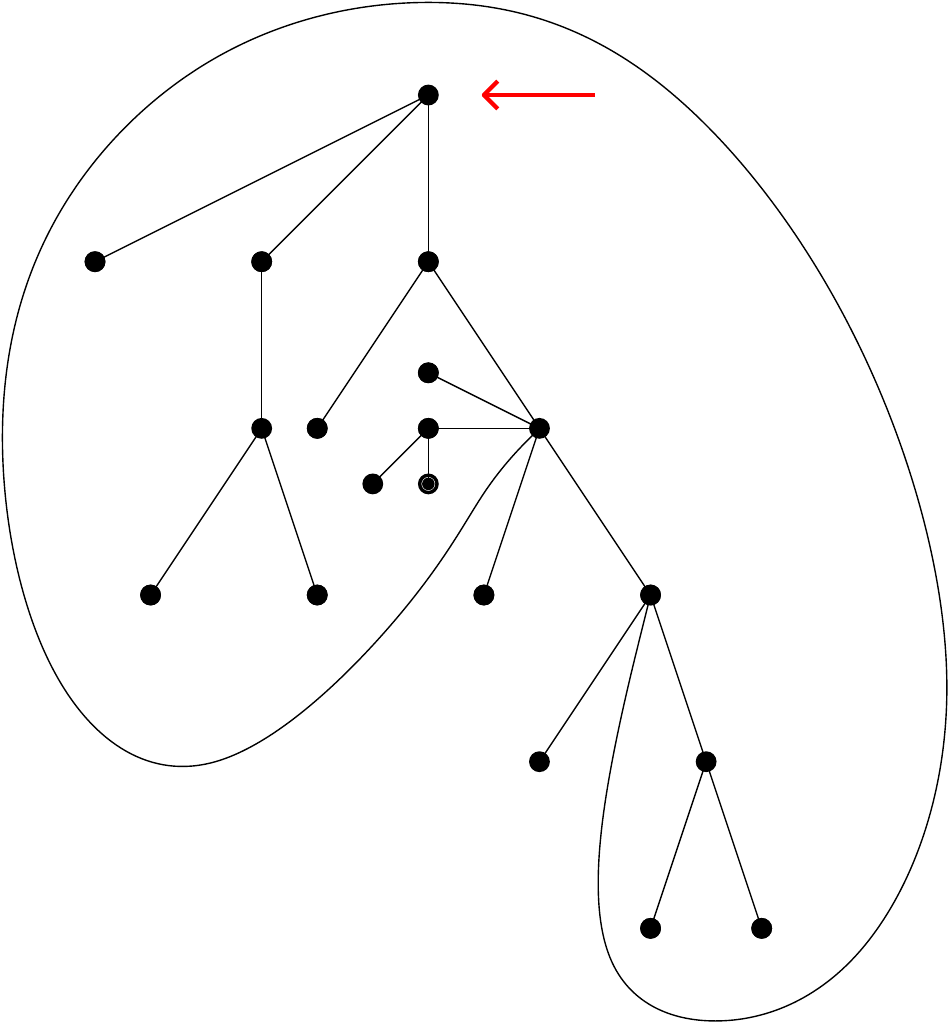}
\end{center}
\caption{Construction of $\Phi(A,s,6)$}
\label{fig:bij2}
\end{figure}

\begin{figure}[ht]
\begin{center}
   \includegraphics[scale=0.45]{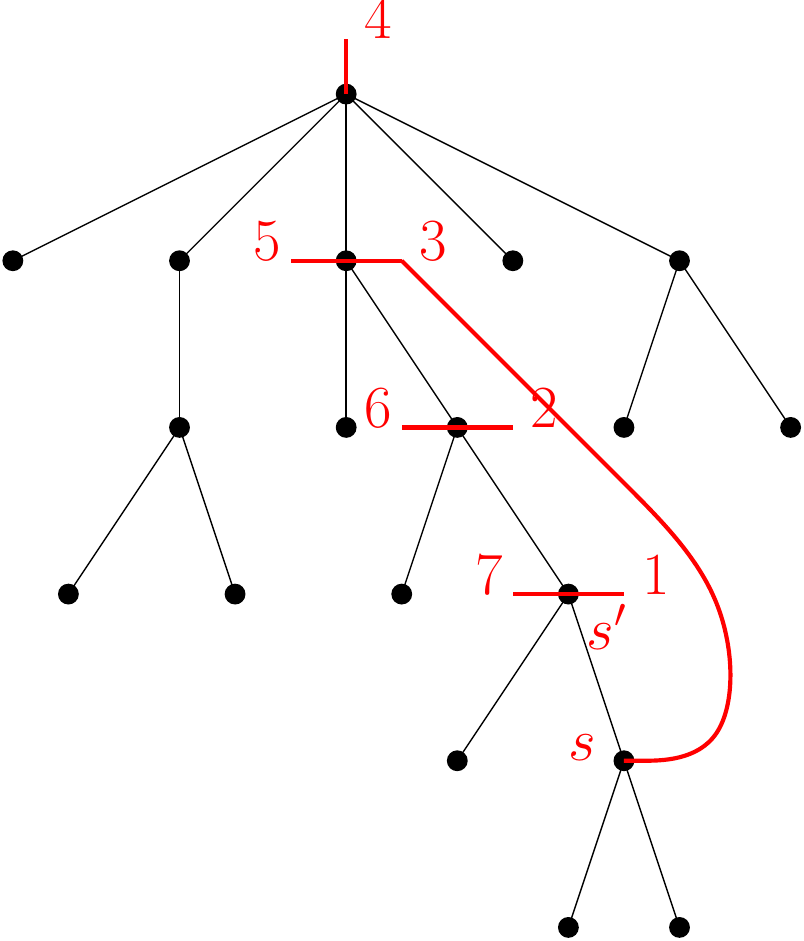}
   $\ \longrightarrow \ $
   \includegraphics[scale=0.45]{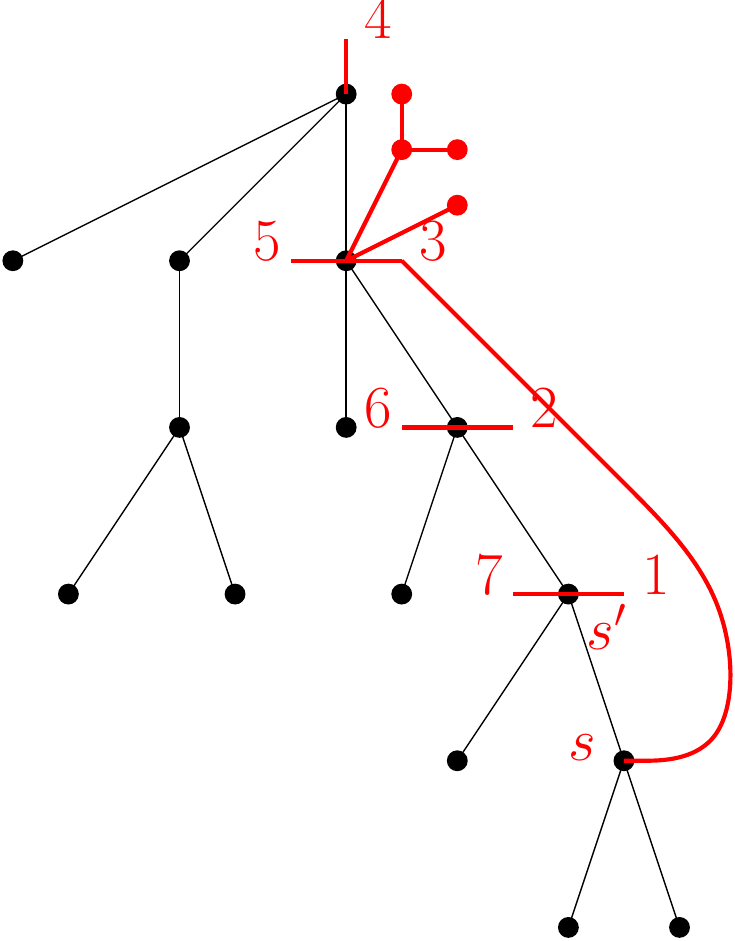}
   $\ \longrightarrow \ $
   \includegraphics[scale=0.45]{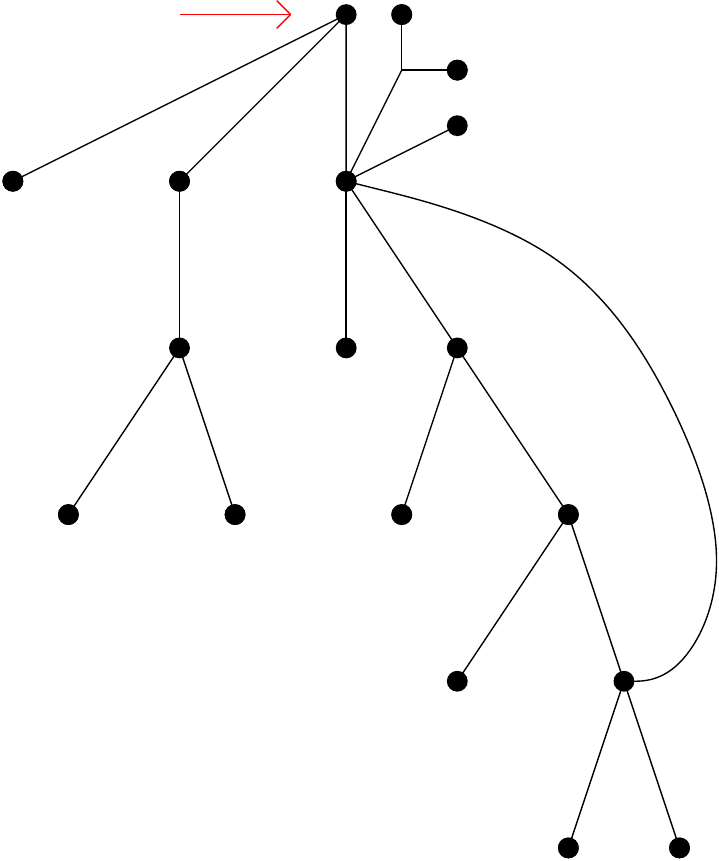}
\end{center}
\caption{Construction of $\Phi(A,s,3)$}
\label{fig:bij3}
\end{figure}

\bigskip

To show that $\Phi$ is a bijection, we shall construct the reverse bijection.
{\em This construction is only sketched below.}
Let us consider an element $(B,r)$ of $\B_n$.
We first ``suspend'' it by the (future) root $r$. 
More formally, each vertex may be assigned a depth through its distance from $r$; this is non ambiguous since the cycle is planar: we follow its edges counterclockwise.
When doing this, we place its son which leads to the cycle to the right.

Then we get back the right subtree of $r$ (two cases are to be distinguished according 
to whether $r$ lies in the inner or outer face of the cycle).
Next we get back $s$ and $i$ in the following way.
We consider $u$ the lowest vertex of the cycle, and: 

    $\bullet$  if $u$ has a subtree inside the cycle, then $u=s'$, and $s$ is its
        leftmost son that lies in the inner face, and $i$ is even,
        
    $\bullet$  if not, then $u=s$ and $i$ is odd.
\end{proof}

\section{Intrinsic thickness and enumeration}\label{sec:periodic}

In this section we study the relations between Periodic Parallelogram Polyominoes having the same list of trees through the bijection of Theorem~\ref{thm:bijquadruplet} but with different intrinsic thicknesses. The principal result is given in Theorem~\ref{thmaire}, it has consequences in terms of periodicity of generating functions and leads us to introduce the notion of primitive PPP. Those are enumerated in Theorem~\ref{thmprimitive}.

\begin{Theorem}\label{thmaire}
Let $P_1$ and $P_2$ be two PPP's such that their images through the bijection of Theorem~\ref{thm:bijquadruplet} contain exactly the same list of $4$-tuples of trees, but have different intrinsic thickness. Assume that the difference between their intrinsic thicknesses is the width of their common trunk PPP.  Then we have:

\smallskip
$(i)$ $P_1$ and $P_2$ have the same upper (resp. lower) path, and have the same cyclic structure,

\smallskip
$(ii)$ the difference between the areas of $P_1$ and $P_2$ is $w\times h$, where $w$ is their common width and $h$  their common height.

\end{Theorem}

\begin{figure}[!ht]
\begin{center}
\includegraphics[scale=1]{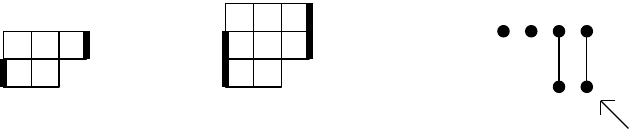}
\caption{Illustration of Theorem~\ref{thmaire}. The list of $4$-tuples of trees is on the right, and the width of the common PPP-trunk is $1$. The  PPP's have intrinsic thickness $1$ and $2$.}
\end{center}
\end{figure}

\begin{proof}
We proceed by induction on the number of non-root vertices in the list of $4$-tuples of trees. If this number is zero, then the two PPP's are staircase PPP's, with the same width and height, and checking properties about paths and areas is immediate. The property about cyclic structures comes directly from the hypothesis about the difference between their intrinsic thicknesses.

Let $P_1$ and $P_2$ be two PPP's satisfying the hypothesis. Denote by $L$ their common list of $4$-tuples of trees and choose $u$ one leaf of greatest depth of one tree in $L$. Let $P_3$ (\emph{respectively} $P_4$) be the PPP with the same intrinsic thickness as $P_1$ (\emph{respectively} $P_2$) and whose corresponding list (denoted by $L'$) is $L$ in which we delete the leaf $u$. By induction hypothesis, $P_3$ and $P_4$ satisfies $(i)$ and $(ii)$. 

We encode the upper (\emph{respectively} lower) path of $P_i$ (for $i \in \{1,\ldots,4\}$) by a finite binary word $w_{u,i}$ (\emph{respectively} $w_{\ell,i}$) by encoding  each horizontal step of the path by $1$, each vertical step of the path by $0$ and reading this (0-1)-encoding from bottom to top and left to right. In this encoding, each black vertex  in $L$ corresponds to a $0$ in the upper binary word and each white vertex corresponds to an $1$ in the lower binary word. By induction hypothesis, we have $w_{u,3}=w_{u,4}$ and $w_{\ell,3}=w_{\ell,4}$.

We now want to describe the impact on binary words of adding the vertex $u$ to the list $L'$. We will show that this impact depends only on the binary words and on the cyclic structure (and does not depends on the intrinsic height). This will allow us to prove $(i)$.

Denote by $v$ the father of $u$. Assume first that $u$ is a white vertex (i.e we want to add a column to the PPP). Adding $u$ to the list $L'$ is done in the following way: determine the $0$ in the upper binary word corresponding to $v$ and replace it by $01$. Next, according to which tree contains $v$, find the $1$ in the lower binary word corresponding to the column located directly at the right of the $0$ in the upper binary word, or the $1$ in the lower binary word corresponding to the column located at the end of the sequence of $1$ following the $0$ in the upper binary word and replace this $1$ by $11$. If $u$ is a black vertex, the same description also holds, after exchanging the roles of $0$ and $1$ and changing column by line. All this process depends only on the considered binary words and on the cyclic structure of the PPP, and does not depend on the intrinsic thickness. Thus, we make the same changes on $w_{u,3}$ and $w_{u,4}$ and on 
$w_{\ell,3}$ and $w_{\ell,4}$. As these words are equal by induction hypothesis, we also have $w_{u,1}=w_{u,2}$ and $w_{\ell,1}=w_{\ell,2}$. Moreover, all the operations preserve the cyclic structure. Then, we have proved $(i)$.

As $P_3$ and $P_4$ have the same lower and upper paths, we can obtain the one with greatest area (assumed now to be $P_4$) by adding a fixed number of boxes to each column of the other, namely $h(P_3)$ according to the induction hypothesis $(ii)$. As we already seen, $P_1$ and $P_2$ also have the same lower and upper paths, so we can obtain the one with greatest area by adding a fixed number of boxes to each column of the other. 

In the case where  we add a column to $P_3$ to obtain $P_1$, we can directly conclude that this number of boxes is the same for transforming $P_3$ to $P_4$ and for transforming $P_1$ to $P_2$, namely $h(P_3)$, which is also $h(P_1)$. The property $(ii)$ follows in this case.

In the case where we add a line to $P_3$ to obtain $P_1$, it is not direct but we can show that the number of boxes by column necessarily to transform $P_1$ to $P_2$ is equal to one plus the number of boxes necessarily to transform $P_3$ to $P_4$. So we need $h(P_3)+1=h(P_1)$ boxes by column. Property $(ii)$ follows in this case, and this concludes the proof. \end{proof}

We now want to emphasize the fact that if we have two PPP's satisfying the hypothesis of the previous theorem, the one with greatest area  can be obtained by adding $h$ boxes to each column of the other (where $h$ is the height of the two PPP's), and marking the result such that the two PPP have the same semi-perimeter. If a PPP $P_1$ can be obtained from another one PPP $P_2$ through this process, we say that $P_1$ \emph{derives} from $P_2$. This leads us to distinguish a subset of PPP's, the ones which can not be derived from another one.

\begin{definition}\label{defprimitif}
    Let $P$ be a PPP. $P$ is \emph{primitive} iff its intrinsic thickness is
    less or equal to the width of its trunk PPP.
\end{definition}

\begin{corollary}
    Let $n$ be an integer. Let $A_n(q)$, $B_n(q)$ and $C_n(q)$ be respectively
    the generating function of PPP's, marked PPP's and strips of fixed
    semi-perimeter according to the area.
\begin{equation}A_n(q):= \sum_{P \in PPP,~sp(P)=n} q^{area(P)}.
\end{equation}
The coefficients of $A_n(q)$, $B_n(q)$ and $C_n(q)$ are ultimately periodic. In
all three cases, period divides the least common multiple of the integers $i(n-i)$
for $i \in \{1, \dots, n-1\}$.
\end{corollary}

\begin{proof}
We first consider the generating function $A_n(q)$. According to Theorem~\ref{thmaire} $(ii)$, we can gather PPP's depending on which primitive PPP they derive to obtain:
\begin{equation}\label{eqperiod}
A_n(q) = \sum_{P \in \text{ primitive PPP}} q^{area(P)} \left(1+q^{w(P)h(P)}+q^{2w(P)h(P)}+\cdots \right),
\end{equation}
where $w(P)$ and $h(P)$ denotes the width and height of $P$.

According to Theorem~\ref{thmprimitive} below, there is only a finite number of
primitive PPP's with semi-perimeter $n$. Therefore, in \eqref{eqperiod}, we
write $A_n(q)$ as a finite sum of series with ultimately periodic coefficients.
This implies that the coefficients of $A_n(q)$ are ultimately periodic and  the
period is a divisor of the least common multiple of all the periods of those
series, namely all the integers $w(P)h(P)$, where $P$ is a primitive PPP. As we
have  $w(P)+h(P)=n$, the result follows for $A_n$.

When $P_1$ derives from $P_2$, they have the same number of boxes in the first
column that we can choose for the second mark, moreover, the derivation
commutes with the rotation, hence, the same proof still stands for $B_n$ and
$C_n$.
\end{proof}

In the case of PPP's, the previous result is new. In the case of marked PPP's, it was already known that the coefficients of $B_n$ are ultimately periodic, but the knowledge concerning period is that the period divides $n$ \cite[Theorem 2.3]{BJN2}. Mixing this with our result allows us to re-obtain the following result \cite[Corollary 4.1]{HJ}.

\begin{corollary}
Let $n$ be a prime integer. The period of the coefficients of the generating function of marked PPP $B_n(q)$ is 1.
\end{corollary}

\begin{proof}
We already known that the sequence of coefficients of $B_n(q)$ admits both periods $n$ and the least common multiple of the integer $i(n-i)$ for $i \in \{1, \dots, n-1\}$. So it admits a period equal to the greatest common divisor of these two numbers. When $n$ is prime, this gcd is 1.
\end{proof}

We now focus on primitive PPP's and achieve their enumeration according to semi-perimeter.

\begin{Theorem}\label{thmprimitive}
Let $n$ be an integer. The number of primitive PPP's with semi-perimeter $n$ is $\displaystyle \frac{(2n+1)!}{n!^2}$. Their generating series is $\displaystyle pPPP(z):=\sum_{P \in \text{ primitive~PPP}} z^{sp(P)}=z^2(1-4z)^{-3/2}$.
\end{Theorem}

\begin{proof}
According to the bijection of Theorem~\ref{thm:bijquadruplet} and the
Definition~\ref{defprimitif}, the set of primitive PPP's with semi-perimeter
$n$ is in bijection with the set of lists of length $k$ (where $k$ is an
arbitrary integer) of $4$-tuples of trees and an integer between $1$ and $k$
(where this integer is the intrinsic thickness). In terms of generating
function, using the computation of generating functions of such 4-tuples of
trees done in \cite[Equation 4]{BLZ}, this leads to:
\begin{equation}
pPPP(x)=\displaystyle \sum_{k \geq 1}k z^{2k}  A(z)^{4(k-1)} \frac{{(\sqrt{1-4z} - 1)}^{4}}{16 \, z^{4} \sqrt{1-4z}},
\end{equation}
where $A(z)$ is the well-known generating function of planar tree according to
the number of vertex. Replacing the term $z^{2k}$ by $u^k$ allows us to compute
the sum after recognizing that this sum is now the partial differentiation with
respect to variable u of a geometric sum. We then replace $u$ by $z^2$ to
obtain:
\begin{equation}
pPPP(x)=\frac{{(\sqrt{1-4z} - 1)}^{4}}{16 \, z^{4} \sqrt{1-4z}} \frac{z^2}{(1-z^2A(z)^4)^2}. 
\end{equation}
As an exact expression for $A(z)$ is known, a straightforward computation leads
to the announced expression of $pPPP(x)$. The number of primitive PPP's comes
from coefficient extraction.
\end{proof}


Proposition~\ref{pop_dyck_area} and Theorem~\ref{thmprimitive} are very close.
Indeed, PPP's with intrinsic thickness equal to 1 form a subset of primitive
PPP's. The first family is enumerated by sequence A008549 of \cite{oeis}, which
counts the total (triangular) area under Dyck paths of fixed size, while the
second is enumerated by the sum of the area of the triangles obtained by
extending the triangular peaks of Dyck paths of fixed size
(\cite[A002457]{oeis}). A bijective proof of the first result is given at
Section~\ref{sec:bijPPP}. A similar bijection for
Proposition~\ref{pop_dyck_area} (which could unify the two results) is still
missing.

The sequence \cite[A008549]{oeis} counts also the number of edges in the Hasse
diagram of the poset of partitions contained in the $n\times n$ box and ordered by
containment. Unexpectedly, a subset of primitive PPP's, thin PPP's, counts
those partitions, minus the empty one (\cite[A030662]{oeis}). A
\emph{thin PPP} is a PPP such that one column contains only vertex cells, in
particular a thin PPP is of intrinsic height 1.

Finally, a computer exploration using Sage~\cite{sage} tells us until
semi-perimeter 8 that the number of marked primitive PPP's is equal to twice
the number of primitive PPP's.

\section{Asymptotic of strips}\label{sec:asymptotic}

 In this section, we study asymptotic estimate of the coefficients of the generating function $B(z)$ of strips with intrinsic thickness 1 (equivalently $i$, with $i$ an arbitrary number) according to their semi-perimeter. 
 
We will use here classical methods in asymptotic theory, which can mostly be
found in \cite{FS}. Recall that the fundamental idea is that the exponential
growth of the coefficients is determined by the dominant singularity of the
generating function (which is analytic at the origin), \textit{i.e}
singularities at the boundary of the disc of convergence,  while the
subexponential factor can be computed by studying the type of this
dominant singularity (for example, the order of the poles). We state here the
singular expansion in the case of  algebraic-logarithmic singularity
\cite[Theorem VI.6, Equation (27)]{FS} that we use latter.

\begin{Theorem}\label{alglog}
Let $\alpha$ and $\mu$ be two positive integers. Then the coefficients $f_n$ of $f(x)=(1-\mu x)^{-\alpha}\log^ k(\frac{1}{1-\mu z})$ admit the following asymptotic expansion in descending power of $n$: \begin{equation}
f_n =[x^n]f(x) :=  \mu^n n^{\alpha-1} \left[F(\log(n))+O\left(\frac1{\log(n)}\right)\right],
\end{equation}
where $F$ is an explicitly computable polynomial with degree $k-1$ and $O$ is the Landau notation.
\end{Theorem}

 We now state the principal result of this section.

\begin{Theorem}
Let $b_n:= [z^n]B(z)$ be the coefficients of $B(z)$. Then they admit the following asymptotic estimation in  descending power of $n$ :
\begin{equation}
b_n= \frac{ 4^n}{2n} \left[1+O\left(\frac{1}{log(n)}\right)\right],
\end{equation}

\end{Theorem}
\begin{proof}
Recall that an expression of $B(z)$ in terms of an infinite sum is already known from \cite[Equation (2)]{BLZ}, and this expression is the following:
\begin{equation}
B(z)=-\displaystyle \sum_{i \geq 1}\displaystyle \frac{\varphi(i)}{i} \log \left(1-\frac{(1-\sqrt{1-4z^i})^4}{16z^{2i}}\right),
\end{equation} 
where $\varphi$ is Euler's totient function.

This insures us that the series $B$ has a single dominant singularity in $z=1/4$ and then we can write:
\begin{equation}
B(z)= -{\log \left(1-\frac{(1-\sqrt{1-4z})^4}{16z^{2}}\right)}+ A(z),
\end{equation}
where we can show that $A(z)$ is analytic at the origin and has radius of convergence $\delta>1/4$. Since the logarithmic term is analytic for $|z|<1/4$, we write $\varepsilon=1/4-z$. This leads to :
\begin{equation}
B(z)= -{\log \left(1-\frac{(1-\sqrt{\varepsilon})^4}{1+8\varepsilon+16\varepsilon^{2}}\right)}+ A(1/4+\varepsilon),
\end{equation}
which can be expanded in $\varepsilon$ through
\begin{equation}\label{eqlogB}
B(z)= -\frac{1}{2} \log(\varepsilon)+O(\sqrt{\varepsilon})+A(1/4-\varepsilon)= \frac{1}{2}\log(\frac{1}{1-4z})\left[1+O(1)\right].
\end{equation}

As $B(z)$ has an algebraic-logarithmic singularity, we can now apply a classical theorem of transfer (see for instance \cite[Theorem VI.3]{FS}) and Theorem\,\ref{alglog}  in the case $\mu=4$, $m=0$ and $k=1$ to conclude. The involved polynomial $F$ has degree $0$, thus it is  a constant, which is exactly  $1/2$, coming from the right-hand side of \eqref{eqlogB}.
\end{proof}

\subsection*{Acknowledgments}
\label{sec:ack}

This research was driven by computer exploration using the open-source software
\texttt{Sage}~\cite{sage} and its algebraic combinatorics features developed by
the \texttt{Sage-Combinat} community \cite{Sage-Combinat}.

\printbibliography

\end{document}